\documentclass[preprint,authoryear]{elsarticle}

\usepackage{mathtools}
\usepackage{tikz-cd}
\usepackage{amsmath,amssymb,amsfonts,amsthm,latexsym}
\usepackage{mathrsfs}
\usepackage{xypic}

\journal{Journal of Symbolic Computation}

\usepackage{blkarray}
\usepackage{graphicx}
\usepackage{amsmath,amssymb,amsfonts,latexsym,amsthm}
\usepackage{color}
\usepackage{eucal}
\usepackage{xspace}
\usepackage{hyperref}
\usepackage{tikz-cd}
\usepackage{bm}
\usepackage{algpseudocode,algorithm,algorithmicx}
\usepackage[title]{appendix}
\usepackage{listings}
\usepackage{multicol}
\usepackage{subcaption}

\newtheorem{lemma}{Lemma}

\newtheorem{theorem}{Theorem}
\theoremstyle{definition}
\newtheorem{definition}{Definition}

\usepackage[round]{natbib}

\DeclareMathOperator{\Ima}{Im}
\DeclareMathOperator*{\rnk}{rnk}
\begin{document}

\begin{frontmatter}
	
	
	
	\title{A computational framework for weighted simplicial homology}
	

\author[add1]{Andrei C. Bura}
\ead{cb8wn@virginia.edu}

\author[add2]{Neelav S. Dutta}
\ead{nsd8uc@virginia.edu}

\author[add1]{Thomas J. X. Li\corref{c1}}
\ead{jl9gx@virginia.edu}

\author[add1,add2]{Christian M. Reidys\corref{c1}}
\cortext[c1]{Corresponding authors}
\ead{cmr3hk@virginia.edu}
	
\address[add1]{Biocomplexity Institute \& Initiative,
		University of Virginia, 995 Research Park Blvd,
	Charlottesville, VA, 22911}
\address[add2]{	Department of Mathematics,
	University of Virginia,
	141 Cabell Dr, Charlottesville, VA 22903}
	
	\begin{abstract}
We provide a bottom up construction of torsion generators for weighted homology of a weighted complex over a discrete valuation ring $R=\mathbb{F}[[\pi]]$. This is achieved by starting from a basis for classical homology of the $n$-th skeleton for the underlying complex with coefficients in the residue field $\mathbb{F}$ and then lifting it to a basis for the weighted homology with coefficients in the ring $R$. Using the latter, a bijection is established between $n+1$ and $n$ dimensional simplices whose weight ratios provide the exponents of the $\pi$-monomials that generate each torsion summand in the structure theorem of the weighted homology modules over $R$. We present algorithms that subsume the torsion computation by reducing it to normalization over the residue field of $R$, and describe a Python package we implemented that takes advantage of this reduction and performs the computation efficiently.

		
	\end{abstract}
	
	
	
		

		
	
\end{frontmatter}




\section{Introduction}

Topological Data Analysis (TDA) is a relatively recent field that emerged from applications in algebraic topology and computational geometry. Geometric approaches to data analysis date as far back as~\citep{Roux:05} and TDA was established as a field beginning with the works of~\cite{Edelsbrunner:02} and of~\cite{Zomorodian:04} in persistent homology. This was popularized in a seminal publication~\citep{Carlsson:09}. 
The central dogma of TDA is the idea that topological and geometrical approaches can provide qualitative, and sometimes quantitative information on the structural relations of points within a data-set~\citep{Chazal:17}. The goal thus far of TDA is to formalize, within well-founded mathematical and algorithmic frameworks, analysis methods for point cloud data (metric or Euclidean), that represent the global underlying structure of the data. In recent years, tremendous efforts have been undertaken to provide the data science community at large with efficient data structures and robust algorithms for TDA, by the implementation of standard libraries such as the Gudhi library (C++ and Python)~\citep{Maria:14} and its R software
interface~\citep{Fasy:14}. TDA thus rapidly emerges as a novel and reliable set of tools in the data scientist's arsenal.

From the perspective of applications, TDA has found use with a high degree of success in diverse fields such as  e.g., material science~\citep{Kramar:13,Nakamura:15}; 3D shape analysis~\citep{Skraba:10,Turner:14}; multivariate time series analysis~\citep{Seversky:16}, biology~\citep{Yao:09}, chemistry~\citep{Lee:17}, sensor networks~\citep{Desilva:07} just to name a few. 

Many real world data sets have been shown to exhibit simplicial structure~\citep{Moore:12,Ramanathan:11,Lin:05} and indeed have been organised as such~\citep{Carlsson:09,Spivak:09}. However, a prevalent feature of these data-sets is the presence of additional simplex specific data~\citep{Ebli:20}. At present, such data have not been satisfactorily incorporated or reflected in the investigations of the simplicial structures.

To illustrate this point, we shall have a closer look at research collaboration networks. Such networks can be organised simplicially in a natural way as follows: authors are considered vertices, and a $k$ simplex in-between $k+1$ authors appears if those authors appeared together as authors on a paper (by themselves or amongst others). This question straightforwardly leads to a simplicial complex that can be studied via TDA.

However, a crucial feature of such a network, that cannot be expressed via the simplicial structure is the citation number of a simplex. 
I.e., the number of citations of papers that the $k+1$ authors appeared on together (by themselves or amongst others).

For each simplex, this integer constitutes a weight and a key observation regarding these weights is the fact that the weight of a face $\tau$ of the simplex $\sigma$ is larger than or equal to the weight of $\sigma$. The weight of $\tau$ however does not necessarily divide the weight of $\sigma$ - in fact this is more likely always the case for arbitrary integer weighted complexes. As such the theory put forward by~\cite{Dawson:90,Ren:18} is not immediately applicable here. However, we can amend the theory by considering an appropriate ring - in fact, a discrete valuation ring (DVR) R - where the monotonicity constraint on the weights translates into a divisibility condition.

As a concrete example, consider the following situation $A,B,C,D$ are four authors that have not appeared as co-authors on any papers, but any subset of them has written papers together. Suppose furthermore, that any triplet of authors has been cited once, while any pair of authors has been cited twice except for the pair $A$ and $B$ which has been cited four times in total. Finally, each individual author has been cited five times. Let $X$ denote the simplicial complex obtained by this data-set. Note that $H_1(X)\cong 0$ as this complex is an empty tetrahedra. However, we see that when taking the weights into account and passing to weighted homology over the DVR $R$ with uniformizer $\pi$, the corresponding group will in fact be full torsion, i.e. $H_(1,R)(X)\cong R/(\pi)\oplus R/(\pi)\oplus R/(\pi)$ (see Figure~\ref{F:tetra}). This indicates that the novel DVR homology encodes finer features of the data than simplicial homology alone.

\begin{figure}[h]
	\centering
	\includegraphics[width=0.7\textwidth]{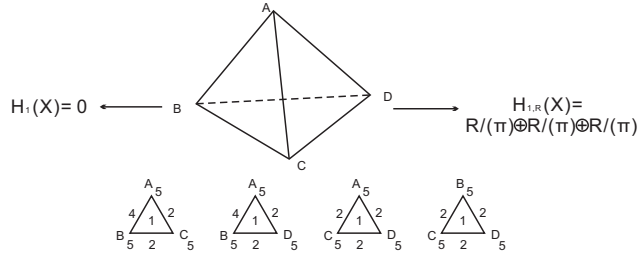}
	\caption
	{\small 
	An example: a weighted triangulation of a sphere. Singular homology yields a trivial $H_1$. However when weights are taken into account, the first weighted homology group is a nontrivial torsion module.
	}
	\label{F:tetra}
\end{figure}

A naive attempt at the computation of weighted homology might proceed via Smith Normalization directly on the weighted boundary matrix with coefficients in rational polynomials localized at the prime ideal $x$. This approach proves quickly intractable due to the storing and processing of the representations of ring elements and the implementation of their algebra. The Smith normalization process (already susceptible to coefficient ballooning) becomes too unwieldy, even for the smallest of complexes (on the order of ten simplices!).

To overcome this challenge we will end up proving that the Smith Normalization over $R=\mathbb{F}[[\pi]]$ is tantamount to the computation of a Normalization over the residue field of the ring $R$ together with the computation of a simplex pairing whose weights will determine the torsion.


Our paper is structured as follows:

Section 2 introduces the definition of a weighted simplicial complex, the weighted boundary operator and the construction of the corresponding weighted homology over a discrete valuation ring $R$ with uniformizer $\pi$.

Section 3 provides a bottom up construction of torsion generators for the weighted homology of a weighted complex over the prototype ring $R=\mathbb{F}[[\pi]]$. This is achieved by starting from a basis for classical homology of the $n$-th skeleton for the underlying complex with coefficients in the residue field $\mathbb{F}$ and then lifting it to a basis for the weighted homology with coefficients in the ring $R$. Using the latter, a bijection is established between $n+1$ and $n$ dimensional simplices whose weight ratios provide the exponents of the $\pi$-monomials that generate each torsion summand in the structure theorem of the weighted homology module over $R$.

Section 4 describes two algorithms that subsume the torsion computation elucidated in Section 3, as well as some remarks on computational improvement over classical Smith Normalization methods, yielded from considerations on the construction of the bijection in Section 2.

Section 5 describes our Python package that implements the algorithms in Section 4 as well as some technical challenges we had to overcome during the implementation process.


\section{Preliminaries}




\subsection{Weighted simplicial homology}


In ~\cite{Dawson:90,Ren:18} the notion of a weighted simplicial complex was introduced and studied. Below we provide a specialized definition and construction for the homology of such complexes, following~\cite{Bura_weighted_21,Li:22}.
\begin{definition}
	A \emph{weighted simplicial complex}
	is a pair $(X,\omega)$ consisting of a simplicial complex $X$ and a non-negative integer weight function
	function $\omega : X \rightarrow \mathbb{N}$ such that for any simplices $\sigma,\mu \in X$ we have
	\[
	\sigma\subseteq \mu \implies \omega(\sigma)\geq \omega (\mu).
	\]
\end{definition}
Let $R$ be a discrete valuation ring (DVR) with uniformizer $\pi$. Then, $(X,\omega)$ naturally induces a (tropical) simplex weight function $v:X\rightarrow R$ by setting $v(\sigma)=\pi^{\omega(\sigma)}$. In this paper we use $(X,\omega)$ and $(X,v)$ interchangeably to denote a weighted simplicial complex. The \emph{weighted chain complex} $C_{n}(X,R)$ is defined to be the free $R$-module generated by the set of $X$ $n$-simplices, and we can set the \emph{weighted boundary map} to be, 
\[ 
\partial^v_n:C_{n}(X,R)\to C_{n-1}(X,R),\quad\quad
\partial^v_n(\sigma)=\sum_{i=0}^n  \frac{v(\hat{\sigma}_i)}{v(\sigma)}\cdot (-1)^i\hat{\sigma}_i,
\]
where $\frac{v(\hat{\sigma}_i)}{v(\sigma)}:=\pi^{\omega(\hat{\sigma}_i) -\omega(\sigma)}$.
The \emph{weighted homology} $H_{n}^v(X)$ of  $(X,v)$ is then the sequence of $R$-modules $H_{n}^v(X) =\ker \partial_{n}^v / \text{\rm Im } \partial_{n+1}^v $. 

\subsection{The chain maps $\theta_n^{v,v'}$}
Given two weighted complexes $(X,v')$ and $(X,v)$, we can define chain maps
\[
\theta_n^{v',v}\colon C_{n}(X,R)\to C_{n}(X,R),\quad\quad
	\theta_n^{v',v}(\sigma)=\frac{v(\sigma)}{v'(\sigma)} \sigma. 
\]
By abuse of notation we write $\theta_n=\theta_n^{v',v}$. These induce natural homomorphisms
\[\bar\theta_n\colon H_{n}^{v'}(X) \to H^v_{n}(X),\quad\quad \bar\theta_n(\sum_ja_j\sigma_j+\text{\rm Im }\partial^{v'}_{n+1}):=
\theta_n(\sum_ja_j \sigma_j) +\text{\rm Im }\partial^v_{n+1}.
\]
Again, abusing notation we will write $\theta_n$ instead of $\bar\theta_n$ throughout the remainder of the paper.

\section{The explicit structure of the weighted simplicial homology $H_{n}^v(X)$}

In this section we provide a bottom up construction of the torsion generators in the weighted homology module for a fixed weighted complex. The model for our DVR going forward will be a power series ring $R=\mathbb{F}[[\pi]]$ with residue field $\mathbb{F}$.

\begin{lemma}\label{L:basis}
Let $(X,v)$ be a weighted complex. There exists $K\subsetneq \{\sigma\mid \sigma\in C_{n}(X)\}$ and a distinguished $K$-basis $\mathcal{B}_K=\{\beta_\kappa | \kappa\in K\}$ for $H_n(X^n;\mathbb{F})$ such that \\
(i)  $M=\complement{K}$ is a basis of $\partial_n(C_{n}(X;\mathbb{F}))$,\\
(ii) each $\beta_\kappa$ contains a  distinguished  $n$-simplex $\kappa \in K$,  having coefficient one,\\
(iii) $\kappa$ appears exclusively in $\beta_\kappa$ and has minimum weight $\omega$ among all $\beta_\kappa$-simplices,\\
(vi) $\mathcal{B}_K=\{\beta_\kappa | \kappa\in K\}$ is also a basis for $H_n(X^n;R)$.
\end{lemma}
\begin{proof}
We recursively construct an Ansatz for $\mathcal{B}_K$ as follows:

Let the $n$-simplices in $X$ be ordered $\sigma_1,\sigma_2,\ldots,\sigma_l$ such that $\omega(\sigma_{i})\geq \omega(\sigma_{i+1})$ for all $1\le i\le l-1$. We first construct a sequence $\{M_m\}_{m\ge 0}$ of sets of $n$-simplices by examining the above $n$-simplices in order one by one. We set $M_0=\varnothing$, $M_1=M_0\cup \{\mu_1\mid \mu_1=\sigma_1\}$ and, given $M_m$ already constructed, $M_{m+1}=M_m \cup \{\mu_{m+1}=\sigma_{m+1}\}$ if 
\begin{equation}\label{Eq:beta}
\partial_n   (\sum_{\mu_j\in M_m}r_j \mu_j + r_{m+1}\sigma_{m+1})=0
 \end{equation}
has no nontrivial solution $r_j \in \mathbb{F}$. Otherwise, set $M_{m+1}=M_m$. Then, $\partial_n(M_m)\subset \partial_n(M_{m+1})$ and when the algorithm terminates, we obtain the $n$-simplex bi-partition $M\dot\cup K$ for $M=M_l$ and $K:=\complement{M}$, where by construction, $M$ is a basis of $\partial_n(C_{n}(X;\mathbb{F}))$. 

Furthermore, for each $\kappa=\sigma_{m+1} \in K$,
\[
\beta_\kappa=\sum_{\mu_j\in M_m}q_j \mu_j +\kappa
\]
is a $H_n(X^n;\mathbb{F})$-cycle. Note that $\beta_\kappa$ is unique. Indeed, otherwise 
\[
\partial_n(\sum_{\mu_j\in M_m}q_j \mu_j -\sum_{\mu_j\in M_m}q'_j \mu_j)=0, \]
which is a contradiction since $\sum_{\mu_j\in M}(q_j-q'_j) \mu_j\neq 0 $ cannot be a $H_n(X^n;\mathbb{F})$-cycle by definition. Thus, $\kappa$ appears exclusively in $\beta_\kappa$ and, due to the order in which the $n$-simplices are processed, has minimum $\omega$ weight among the $\beta_\kappa$-simplices.

Now, since each $\kappa$ uniquely corresponds to a $\beta_\kappa$, $\mathcal{B}_K$ are linearly independent over $\mathbb{F}$. To prove $\mathcal{B}_K$ is a basis, it then suffices to show that $\langle\mathcal{B}_K\rangle_\mathbb{F}=H_n(X^n;\mathbb{F})$. As previously noted, any cycle $c=\sum a_j \sigma_j\in H_n(X^n;\mathbb{F})$ must contain at least one $\kappa\in K$. Let $c_0=c-\sum_{\kappa \in K} a_\kappa \beta_\kappa$ be $c$ with all its $\kappa$ simplices eliminated. Then the cycle $c_0$ contains only simplices in $M$. But then $c_0=0$ and so $c=\sum_{\kappa \in K} a_\kappa \beta_\kappa$.

For the local ring $R$, $\pi R$ is the maximal ideal of  $R$, and the quotient of  $H_n(X^n;R)/ \pi  H_n(X^n;R)$ is isomorphic to $H_n(X^n;\mathbb{F})$ due to the structure theorem of finitely generated modules over principal ideal domains.
In view of Nakayama's Lemma\footnote{Let $M$ be a finitely generated module  over a local ring $R$ with maximal ideal $m$. Then  each basis of $M/m M$ lifts to a minimal set of generators of  $M$.}, if 
$\mathcal{B}_K=\{\beta_\kappa | \kappa\in K\}$ were a basis for $H_n(X^n;\mathbb{F})$ then it would also be a basis for $H_n(X^n;R)$,
completing the proof.
\end{proof}

\begin{lemma}\label{L:partialbasis}
	Let $(X,v)$ be a weighted complex and suppose that $\mathcal{B}_K=\{\beta_\kappa | \kappa\in K\}$ is
a basis of $H_n(X^n;\mathbb{F})$ given by Lemma~\ref{L:basis}, and $M:=\complement{K}$.
	Then  \\
(1)	$\{\partial_{n}(\mu_j)| \mu_j \in M\}$  is a basis of $\partial_{n} (C_{n}(X;\mathbb{F}))$.\\
(2)	$\{\partial_{n}^v(\mu_j)| \mu_j \in M\}$  is a basis of $ \partial_{n}^v (C_{n}(X;R))$.\\
\end{lemma}

\begin{proof}
Item $(1)$ is immediate from Lemma~\ref{L:basis}. For item $(2)$, by definition, $\langle\{\partial_{n}^v(\sigma)| \sigma \in C_n(X)\}\rangle_R=\partial_{n}^v (C_{n}(X;R))$.  In view of Lemma~\ref{L:basis} and $\partial_{n}^v\circ\theta_{n} = \theta_{n}\circ\partial_{n}$, we have
\[
0=  \theta_{n}  \circ \partial_{n}(\beta_\kappa)=\partial_{n}^v\circ \theta_{n} (\beta_\kappa) =
\sum_{\mu_j \in M_m } q_j v(\mu_j)  \partial_{n}^v (\mu_j) + v(\kappa) \partial_{n}^v(\kappa).
\]
Since $\omega(\kappa)$ is minimum among $\beta_\kappa$-simplices,
$
\partial_{n}^v(\kappa) =- \sum_{\mu_j\in M_m } \frac{q_j v(\mu_j)}{v(\kappa) }  \partial_{n}^v (\mu_j),
$
which shows that $\partial_{n}^v(C_{n}(X;R))=\langle\{\partial_{n}^v(\mu_j)|\mu_j \in M\}\rangle_R$. It remains to show that $\{\partial_{n}^v(\mu_j)| \mu_j \in M\}$ are $R$-linearly independent.
Suppose there exist $r_j\in R$ such that $\sum_j r_j \partial_{n}^v(\mu_j)=0$ and let $\{\sigma_i\}_i$ denote the set of $n$-simplices. Then, 
\[
\sum_j r_j \partial_{n}^v(\mu_j) = \sum_j r_j  \sum_{\{i| \sigma_i  \subset  \mu_j \}} (-1)^{c_{i,j}} \frac{v( \sigma_i )}{v(\mu_j )}  \sigma_i=\sum_i  \sigma_i\sum_{\{j| \sigma_i  \subset  \mu_j \}} (-1)^{c_{i,j}} \frac{r_j v(\sigma_i)}{v(\mu_j )}.
\]
And so for fixed $i$,
\[
\sum_{\{j| \sigma_i  \subset  \mu_j \}} (-1)^{c_{i,j}} \frac{r_j v(\sigma_i)}{v(\mu_j )} =0.
\]
Letting $r_j=\sum_{l=0}^{\infty} h_{j,l} \pi^l$ with $h_{j,l}\in \mathbb{F}$,
\begin{align*}
& \sum_{\{j| \sigma_i  \subset  \mu_j \}}  \sum_{l=0}^{\infty} (-1)^{c_{i,j}}  h_{j,l}  \pi^{l+\omega(\sigma_i )-\omega(\mu_j )} =0\\
\Leftrightarrow\quad & \sum_{k}^{\infty}  \pi^k  
\sum_{\{j| \sigma_i  \subset  \mu_j \}}   (-1)^{c_{i,j}}  h_{j,k -\omega(\sigma_i )+\omega(\mu_j )}  =0\\
\Leftrightarrow\quad & \sum_{\{j| \sigma_i  \subset  \mu_j \}}   (-1)^{c_{i,j}}  h_{j,k -\omega(\sigma_i )+\omega(\mu_j )}  =0 \qquad \text{for any } k. 
\end{align*}
For each $k$,
\begin{align*}
\sum_j h_{j,k -\omega(\sigma_i ) +\omega(\mu_j )}  \partial_{n}(\mu_j)&=\sum_j h_{j,k -\omega(\sigma_i )+\omega(\mu_j )} \sum_{\{i| \sigma_i  \subset  \mu_j \}} (-1)^{c_{i,j}}\sigma_i  \\
&=
\sum_i  \sigma_i   \sum_{\{j| \sigma_i  \subset  \mu_j \}}   (-1)^{c_{i,j}}  h_{j,k -\omega(\sigma_i )+\omega(\mu_j )}  \\
&=0 .
\end{align*}
Since $\{\partial_{n}(\mu_j)| \mu_j \in M\}$ are $\mathbb{F}$-linearly independent, $h_{j,k -\omega(\sigma_i ) +\omega(\mu_j )} =0 $ for any $k,j$,  yielding $r_j=0$ as desired.
\end{proof}

\begin{definition}
Let $c=\sum_{j\in J} r_j \sigma_j\in H_n(X^n;\mathbb{F})$ be a cycle such that $r_j \neq 0 $ for any $j\in J$. Then $J$ is called the \emph{support} of $c$. Furthermore we define
\[
\hat{c}=\frac{\theta_n(c)}{\text{gcd}(\theta_n(c))}= \sum_{j\in J} r_j \frac{v(\sigma_j )}{\text{gcd}(\theta_n(c))}  \sigma_j ,
\]
where $\text{gcd}(\theta_n(c))=\text{gcd}\{v(\sigma_j)\}_{j\in J}= \pi^{\min_{j\in J} \omega(\sigma_j)}$.
\end{definition}

Note that $\hat{c}\in H_{n}^v(X^n)$, since $\partial_n^v (\hat{c})=\frac{\partial_n^v \circ \theta_n(c)}{\text{gcd}(\theta_n(c))}  = \frac{\theta_{n-1}\circ \partial_n (c)}{\text{gcd}(\theta_n(c))} =0$.

\begin{theorem}\label{T:basis} 
Let $(X,v)$ be a weighted complex and suppose $\mathcal{B}_K=\{\beta_\kappa | \kappa\in K\}$ is a basis of  $H_n(X^n;\mathbb{F})$ given by Lemma~\ref{L:basis}. Then $\hat{\mathcal{B}}_K=\{\hat{\beta}_\kappa | \kappa\in K\}$ is a basis for $H_{n}^v(X^n)$. Furthermore $\hat{\mathcal{B}}_K$ satisfies \\
(i) each $\hat{\beta}_\kappa $ contains a unique, distinguished  $n$-simplex $\kappa \in K$, \\
(ii) $\kappa$ appears exclusively in $\hat{\beta}_\kappa $ and has coefficient one in $\hat{\beta}_\kappa$.
\end{theorem}

\begin{proof}
By construction and Lemma~\ref{L:basis}, items $(i)$ and $(ii)$ are true and furthermore imply $\hat{\mathcal{B}}_K$ are linearly independent. It suffices to show $\langle\hat{\mathcal{B}}_K\rangle_R=H_{n}^v(X^n)$.

Any $H_{n}^v(X^n)$-cycle $c=\sum a_j \sigma_j$ contains at least one $\kappa\in K$.
Let $c_0=c-\sum_{\kappa \in K} a_\kappa \hat{\beta}_\kappa$ be $c$ with all its $\kappa$ simplices eliminated. Then the cycle $c_0$ contains only simplices in $M$, which indicates that $0=\partial_{n}^v(c_0)=\sum_{\mu_j\in M}a_j \partial_{n}^v(\mu_j)$. Since $\{\partial_{n}^v(\mu_j)| \mu_j \in M\}$ are $R$-linearly independent by Lemma~\ref{L:partialbasis},
we have $a_j=0$ for any $j$. Therefore $c_0=0$ and $c=\sum_{\kappa \in K} a_\kappa \hat{\beta}_\kappa$.
\end{proof}

\begin{theorem}\label{T:basis4}
Let $(X,v)$ be a weighted complex. Suppose that $\{\beta_{\kappa_i} | \kappa_i\in K_n\}$ is a basis  for $H_{n}(X^n;\mathbb{F})$ given by Lemma~\ref{L:basis},
and $\{\partial_{n+1}(\mu_j)| \mu_j\in M_{n+1}\}$  is the corresponding basis of $\partial_{n+1} (C_{n+1}(X;\mathbb{F}))$ given by Lemma~\ref{L:partialbasis}. 
Denote by $K_n=\{\kappa_1,\kappa_2,\ldots,\kappa_q\}$  with $\omega(\kappa_i)\leq \omega(\kappa_{i+1})$,
and  $M_{n+1}=\{\mu_1,\mu_2,\ldots,\mu_p\}$  with $\omega(\mu_j)\geq \omega(\mu_{j+1})$.
	Then	there exists a basis of  $H_{n}(X^n;\mathbb{F})$, denoted by $\phi_1,\phi_2,\ldots, \phi_p,\ldots, \phi_q$, satisfying the following conditions\\
	(1)  $\phi_1,\phi_2,\ldots, \phi_p$ is a basis of 
	$\partial_{n+1} (C_{n+1}(X;\mathbb{F}))$,\\
	(2) there exists a partition $K_n=K_1\dot{\cup}K_2$ having $K_1=\{\kappa_{i_1},\kappa_{i_2},\ldots,\kappa_{i_p}\}$ 
	together with a 
	bijection $\Phi: K_1 \to M_{n+1}$ with  $\Phi(\kappa_{i_k})=\mu_{j_k}$  such that 
\[
\{\phi_{p+1},\ldots, \phi_q\} =\{\beta_\kappa| \kappa\in K_2\},
\]
	and for any $1\leq k\leq p$
	\[
\phi_k = \sum_{l\geq i_k} b_{l}\beta_{\kappa_{l}} =\sum_{h\leq j_k} d_h \partial_{n+1}(\mu_h).
	\]
	where $b_l,d_h\in \mathbb{F}$, $b_{i_k}\neq 0$ and $d_{j_k}\neq 0$.
\end{theorem}

\begin{proof}
By hypothesis $\rnk H_{n}(X^n;\mathbb{F})=q$ and $\rnk \Ima \partial_{n+1}=p$. We proceed by examining in order elements from $K_n$ and recursively constructing a sequence $\{B^{(k)}=\{\beta^{(k)}_1,\beta^{(k)}_2,\ldots,\beta^{(k)}_q\}\}_k$ of bases for $H_{n}(X^n;\mathbb{F})$ and a corresponding sequence of bases $\{T^{(k)}=\{t^{(k)}_1,t^{(k)}_2,\ldots,t^{(k)}_p\}\}_k$ for $\Ima \partial_{n+1}$. At each step $k$ we denote $T_1^{(k)}:=T^{(k)}$ and $B^{(k)}:=B_1^{(k)}\dot\cup B_2^{(k)}$ depending on whether the elements are contained in $\Ima \partial_{n+1}$ or not, respectively. Inducting on $k$, we show that $B ^{(k)}$ and $T^{(k)}$ satisfy\\
(i) $B ^{(k)}$ is a basis for $H_{n}(X^n;\mathbb{F})$  and $T^{(k)}$ is a basis for $\Ima \partial_{n+1}$,\\
(ii) $T_1^{(k)}= B_1^{(k)}$,\\
(iii) 	Each $t^{(k)}_j\not \in T_1^{(k)} $ does not contain any $\beta^{(k)}_{i}$-terms for $i\leq k$, i.e., for $q^{(k)}_{i,j}\in  \mathbb{F}$,
\[
t^{(k)}_j=\sum_{i\geq k+1} q^{(k)}_{i,j}\beta^{(k)}_{i}.
\]

The base case:\\
Initialize
$B ^{(0)}=\{\beta^{(0)}_1=\beta_{\kappa_1},\beta^{(0)}_2=\beta_{\kappa_2},\ldots,\beta^{(0)}_q=\beta_{\kappa_q}\} $, $T^{(0)}=\{t^{(0)}_1=\partial_{n+1}(\mu_1),t^{(0)}_2=\partial_{n+1}(\mu_2),\ldots,t^{(0)}_p= \partial_{n+1}(\mu_p)\} $ and set $T_1^{(0)}=B_1^{(0)}=B_2^{(0)}=\varnothing$. Lemma~\ref{L:basis}  guarantees that $B ^{(0)}$ and $T ^{(0)}$ are the bases of $H_{n}(X^n;\mathbb{F})$  and $\Ima \partial_{n+1}$, respectively. Then, for $1\leq j \leq p$
\[
t^{(0)}_j=\partial_{n+1}(\mu_j)=\sum_{\{i|\kappa_i\subset\mu_j\}}q^{(0)}_{i,j}\beta^{(0)}_{i} =\sum_{i\geq 1}q^{(0)}_{i,j}\beta^{(0)}_{i},
\]
where $q^{(0)}_{i,j}=(-1)^{c_{i,j}} \in \mathbb{F}$ if $\kappa_i\subset \mu_j$ and $q^{(0)}_{i,j}=0$ otherwise.

Induction step $k$:\\  
Suppose $B^{(k-1)}$ and $T^{(k-1)}$ are constructed and examine $\kappa_k$. To find out the $t^{(k-1)}_j$ that contains a $\beta^{(k-1)}_k$-term and has the minimum weight, we check  the minimum $j_k$ such that $t^{(k-1)}_{j_k}\not \in T_1^{(k-1)}$ and $q^{(k-1)}_{k,j_k}\neq 0$. We have two cases
\begin{enumerate}
\item  $j_k$ does not exist i.e., $q^{(k-1)}_{k,j}=0$ for $1\leq j \leq p$. Set $B^{(k)}=B ^{(k-1)}$ by $\beta^{(k)}_i =\beta^{(k-1)}_i $ for $1\leq i\leq q$ and $T^{(k)}=T^{(k-1)}$ by $t^{(k)}_j=t^{(k-1)}_j$ for $1\leq j \leq p$. Finally, set $T_1^{(k)}=T_1^{(k-1)},B_1^{(k)}=B_1^{(k-1)}$, $B_2^{(k)}=B_2^{(k-1)}\cup\{\beta^{(k)}_k \}$ and $q^{(k)}_{i,j}=q^{(k-1)}_{i,j}$. By (iii) in the induction hypothesis, $t^{(k-1)}_j=\sum_{i\geq k} q^{(k)}_{i,j}\beta^{(k)}_{i}$. Since $q^{(k)}_{k,j}=q^{(k-1)}_{k,j}=0$ for $1\leq j \leq p$,
$
t^{(k)}_j=t^{(k-1)}_j=\sum_{i\geq k+1} q^{(k)}_{i,j}\beta^{(k)}_{i}.
$
\item $j_k$ exists i.e., $q^{(k-1)}_{k,j_k}\neq 0$. Let $t^{(k)}_{j_k}=t^{(k-1)}_{j_k}$ and $t^{(k)}_{j}=t^{(k-1)}_{j} -  \frac{q^{(k-1)}_{k,j}}{q^{(k-1)}_{k,j_k}} t^{(k-1)}_{j_k} $ for any $j\neq j_k$ (effectively eliminating the $\beta^{(k-1)}_k$-term in $t^{(k-1)}_{j}$). Note, $T^{(k)}$ remains a basis of $\Ima \partial_{n+1}$. For $B ^{(k)}$, let $\beta^{(k)}_i =\beta^{(k-1)}_i$ for $i\neq k$, and replace $\beta^{(k-1)}_k $ by $t^{(k-1)}_{j_k}$, i.e.,
\[
\beta^{(k)}_k =t^{(k-1)}_{j_k}=q^{(k-1)}_{k,j_k} \beta^{(k-1)}_k +\sum_{i\geq k+1} q^{(k-1)}_{i,j_k}  \beta^{(k-1)}_{i}.
\]
Since $q^{(k-1)}_{k,j_k}\neq 0$,  $B^{(k)}$ remains a basis of $H_{n}(X^n;\mathbb{F})$. Finally, let $T_1^{(k)}=T_1^{(k-1)}\cup \{t^{(k)}_{j_k}\}, B_1^{(k)}=B_1^{(k-1)}\cup\{\beta^{(k)}_k \}$, and $B_2^{(k)}=B_2^{(k-1)}$. By construction, $t^{(k)}_{j_k}=\beta^{(k)}_k$ and so $T_1^{(k)}=B_1^{(k)}$. By (iii) in the induction hypothesis, for $t^{(k)}_j\not \in T_1^{(k)}$,
\begin{align*}
t^{(k)}_{j}=t^{(k-1)}_{j} - \frac{q^{(k-1)}_{k,j}}{q^{(k-1)}_{k,j_k}} t^{(k-1)}_{j_k}=\sum_{i\geq k+1} \Big(q^{(k-1)}_{i,j}-\frac{q^{(k-1)}_{k,j} q^{(k-1)}_{i,j_k}}{q^{(k-1)}_{k,j_k}}\Big)\beta^{(k-1)}_{i}.
\end{align*}
Namely, $t^{(k)}_{j}=\sum_{i\geq k+1} q^{(k)}_{i,j}\beta^{(k)}_{i}$ with $q^{(k)}_{i,j}=q^{(k-1)}_{i,j}-\frac{q^{(k-1)}_{k,j}q^{(k-1)}_{i,j_k}}{q^{(k-1)}_{k,j_k}}\in \mathbb{F}$.Note that $j_k$ is the minimum having  $t^{(k-1)}_{j_k}\not \in T_1^{(k-1)}$ and $q^{(k-1)}_{k,j_k}\neq 0$. For those $j$'s satisfying  $t^{(k-1)}_{j} \in T_1^{(k-1)} $ or $j<j_k$, $q^{(k-1)}_{k,j}=0$ and $t^{(k)}_{j}=t^{(k-1)}_{j}$.

\end{enumerate}

The procedure terminates with a basis $B^{(q)}$ for $H_{n}(X^n;\mathbb{F})$ and a basis $T^{(q)}$ for $\Ima \partial_{n+1}$. Let $B_1^{(q)}=\{\phi_1,\phi_2,\ldots, \phi_p \}$ and $B_2^{(q)}=\{\phi_{p+1},\phi_{p+2},\ldots, \phi_q \}$. By construction $B^{(q)}=B_1^{(q)}\dot\cup B_2^{(q)}$ and  $T^{(q)}=T_1^{(q)}=B_1^{(q)}=\{ \phi_1,\phi_2,\ldots, \phi_p \}$. We set $K_1$ to be the set of $\kappa_k$ having Case 2, denoted  by $\{\kappa_{i_1},\kappa_{i_2},\ldots,\kappa_{i_p}\}$. In Case 2, each $\kappa_{i_k}$ corresponds to a unique $\mu_{j_k}$. Thus we establish a bijection between $K_1$ and $ M_{n+1}$ via $\Phi(\kappa_{i_k})=\mu_{j_k}$. Each  $\phi_k$ corresponds to $\beta^{(i_k)}_{i_k} \in B_1^{(q)}$ in case 2, thus
\[
\phi_k=\beta^{(i_k)}_{i_k} =\sum_{l\geq i_k} q^{(i_k-1)}_{l,j_k}  \beta^{(i_k-1)}_{l}=\sum_{l\geq  i_k}  b_l \beta_{\kappa_l}  
\] 
with $b_l= q^{(i_k-1)}_{l,j_k} $ and  $b_{i_k}=q^{(i_k-1)}_{i_k,j_k}\neq 0$. On the other hand, $\phi_k=\beta^{(i_k)}_{i_k} =t^{(i_k-1)}_{j_k}$. Each time we modify $t^{(s)}_{j_k}$, we always subtract
some $t^{(s-1)}_{j}$-term having $j\leq j_k$. Hence 
\[
\phi_k=t^{(i_k-1)}_{j_k}= \sum_{h\leq j_k} d_h t^{(0)}_{h}= \sum_{h\leq j_k} d_h \partial_{n+1}(\mu_h).
\]
By construction, each $\beta^{(k)}_k \in B_2^{(q)}$ must come from Case 1 as such it was not modified, i.e., $\beta^{(k)}_k =\beta^{(0)}_k $ and so $B_2^{(q)}=\{\beta_\kappa| \kappa\in K_2\}$.
\end{proof}

Now we are finally in position to describe the structure of $H_{n}^v(X)$.

\begin{theorem}\label{T:basis6}
Let $(X,v)$ be a weighted complex and let $\{\phi_1,\phi_2,\ldots, \phi_p,\ldots, \phi_q\}$ be a basis for $H_{n}(X^n;\mathbb{F})$ given by Theorem~\ref{T:basis4}. Then $\{\hat{\phi}_1,\hat{\phi}_2,\ldots, \hat{\phi}_p,\ldots, \hat{\phi}_q\}$ is a basis for $H_{n}^v(X^n)$ and $\{\pi^{m_1}\hat{\phi}_1,\pi^{m_2}\hat{\phi}_2,\ldots, \pi^{m_p}\hat{\phi}_p\}$ is a basis for $\partial_{n+1}^v (C_{n+1}(X;R))$.
Furthermore,
\begin{align*}
H_{n}^v(X) &\simeq R^l  \oplus \bigoplus_{k=1}^p R/ (\pi^{m_k}),
	\end{align*}
where $l=\rnk H_{n}(X;\mathbb{F})$, and the invariant factors are given by $\pi^{m_k} = \frac{v(\kappa_{i_k})}{v(\mu_{j_k})}$ with $\Phi(\kappa_{i_k})=\mu_{j_k}$.
\end{theorem}

\begin{proof}
By 	 Theorem~\ref{T:basis4}, we have $	\phi_k = \sum_{l\geq i_k} b_l \beta_{\kappa_{l}} $ for any $1\leq k\leq p$.
In view of $\gcd(\theta_n(\beta_{\kappa} ))=v(\kappa)$ and  $v(\kappa_{i})| v(\kappa_{i+1})$ for any $i$, we have $\gcd(\theta_n(\phi_k )) =\gcd_{l\geq i_k}(\gcd(\theta_n(\beta_{\kappa_{l}}) ))= \gcd_{l\geq i_k}(v(\kappa_{l} ))= v(\kappa_{i_k})  $.

Since $\theta_n(\beta_{\kappa} )=v(\kappa)  \hat{\beta}_{\kappa} $, we derive 
\begin{equation}\label{Eq:base}
\hat{\phi}_k=\frac{\theta_n(\phi_k) }{\gcd(\theta_n(\phi_k )) } = \frac{ \sum_{l\geq i_k} b_l \theta_n(\beta_{\kappa_{l}} )}{v(\kappa_{i_k}) }=b_{i_k} \hat{\beta}_{\kappa_{i_k}} + \sum_{l> i_k} b_l \frac{v(\kappa_{l}) }{v(\kappa_{i_k}) }\hat{\beta}_{\kappa_{l}}.
\end{equation}
As
$\{\phi_{p+1},\ldots, \phi_q\} =\{\beta_\kappa| \kappa\in K_2\}$,
we have 
\begin{equation}\label{Eq:base2}
\{\hat{\phi}_{p+1},\ldots, \hat{\phi}_q\} =\{\hat{\beta}_\kappa| \kappa\in K_2\}.
\end{equation}
Let $f$ be  the linear map from $\{\hat{\beta}_\kappa | \kappa\in K\}$ to $\{\hat{\phi}_i | 1\leq i\leq q \}$.
Then Eq.~\ref{Eq:base} and~\ref{Eq:base2} implies that the transformation matrix of $f$ is upper triangular and with diagonal  $b_{i_1},\dots, b_{i_p},1,\ldots,1$.
Since the matrix of $f$ is non-singular, we have 
$f$  is invertible. 
As $\{\hat{\beta}_\kappa | \kappa\in K\}$ is
a  basis of  $H_{n}^v(X^n)$ due to Theorem~\ref{T:basis}, 
$\{\hat{\phi}_i | 1\leq i\leq q \}$ is also a  basis of  $H_{n}^v(X^n)$.

By  Theorem~\ref{T:basis4}, we have
$
\phi_k =\sum_{h\leq j_k} d_h \partial_{n+1}(\mu_h).
$
We set $$\psi_k =\sum_{h\leq j_k} d_h \frac{v(\mu_{h})  }{v(\mu_{j_k}) }  \partial_{n+1}^v(\mu_h).$$
Let $g$ be  the linear map from $\{\partial_{n+1}^v(\mu_j)| \mu_j \in M_{n+1}\}$ to $\{\psi_k| 1\leq i\leq p \}$.
By arranging $\{\partial_{n+1}^v(\mu_j)| \mu_j \in M_{n+1} \}$,
we derive that  the transformation matrix of $g$ is upper triangular and with diagonal  $d_{1},\dots, d_{p}$.
Since the matrix of $g$ is non-singular, we have 
$g$  is invertible.
As $\{\partial_{n+1}^v(\mu_j)| \mu_j \in M_{n+1}\}$ 
  is a basis of $\Ima \partial_{n+1}^v$ due to  Lemma~\ref{L:partialbasis}, 
  $\{\psi_k| 1\leq i\leq p \}$ is also a  basis of $\Ima \partial_{n+1}^v$.
Utilizing the commutativity  $\partial_{n+1}^v\circ \theta_{n+1} =   \theta_{n}  \circ \partial_{n+1}$, we derive 
\begin{align*}
\hat{\phi}_k &=\frac{\theta_n(\phi_k) }{\gcd(\theta_n(\phi_k) ) }  \\
&=  \frac{ \sum_{h\leq j_k} d_h \theta_n \circ \partial_{n+1}(\mu_h) } {v(\kappa_{i_k}) } \\
&=   \frac{ \sum_{h\leq j_k} d_h \partial_{n+1}^v\circ \theta_{n+1}(\mu_h) } {v(\kappa_{i_k}) } \\
&= \frac{ \sum_{h\leq j_k} d_h  v(\mu_h) \partial_{n+1}^v (\mu_h) } {v(\kappa_{i_k}) } \\
&= \frac{v(\mu_{j_k}) }{v(\kappa_{i_k}) } \psi_k.
\end{align*}
Thus  $\psi_k= \frac{v(\kappa_{i_k})}{v(\mu_{j_k})} \hat{\phi}_k= \pi^{m_k}  \hat{\phi}_k$.
Therefore
\[
H_{n}^v(X) \simeq H_{n}^v(X^n)/ \Ima \partial_{n+1}^v \simeq  \oplus_{k=1}^q \hat{\phi}_{k} R / \oplus_{k=1}^p  \pi^{m_k} \hat{\phi}_{k} R \simeq R^l  \oplus \bigoplus_{k=1}^p R/ (\pi^{m_k}),
\]
where $l=q-p=\rnk H_{n}(X^n;\mathbb{F}) -\rnk \Ima \partial_{n+1} = \rnk H_{n}(X;\mathbb{F}) $.
	\end{proof}

{\bf Remark:} Finding the basis $\{\hat{\phi}_{1},\hat{\phi}_{2},\ldots, \hat{\phi}_q \}$ is tantamount to computing the Smith normal form of the projection matrix on the sub-module $\Ima \partial_{n+1}^v$ under the basis pair $\{\hat{\beta}_\kappa | \kappa\in K_n\}$ and $\{\partial_{n+1}^v(\mu_j)| \mu_j \in M_{n+1}\}$. Taking this perspective,
the procedure in Theorem~\ref{T:basis4} is a variant of the algorithm for the Smith Normal Form (SNF). However, instead of computing the SNF in $R=\mathbb{F}[[\pi]]$, our procedure only requires computations in $\mathbb{F}$, which greatly simplifies the complexity of such an algorithm.
Moreover, in difference to the standard SNF algorithm, which 
requires  choosing a pivot and improving it until it divides all entries in both its row and column, our approach proceeds column by column in the projection matrix and selects the minimum row with non-zero entry for the respective column.
The ordering of the two bases, i.e., the increasing order of $\{\kappa_i\}_i$-weights and the decreasing order of $\{\mu_j\}_j$-weights, guarantees that the selected entry at position $(i_k,j_k)$ has the minimum $\pi$-power in both its row and column,
i.e., the entry is already a viable pivot for the SNF in $R=\mathbb{F}[[\pi]]$. Finally, our procedure provides a combinatorial interpretation of the invariant factors by means of the bijection $\Phi$ linking  the representative $\kappa_{i_k}$ and $\mu_{j_k}$.

\section{Algorithms} 
In this section we present two algorithms for computing weighted homology over a DVR with  model $R=\mathbb{F}[[\pi]]$.
Based on the theoretical developments of the previous sections, the approach is now rather simple: we first compute the basis $\mathcal{B}_K=\{\beta_\kappa | \kappa\in K_n\}$ for $H_n(X^n;\mathbb{F})$ for each dimension $n$, then perform the reduction procedure in Theorem~\ref{T:basis4} to establish  the bijection $\Phi$, see Algorithms~\ref{basis} and~\ref{couple}. The key feature being that all computations are performed directly over the residue field $\mathbb{F}$ and not the full ring. 

\begin{algorithm}
	\caption{$\beta_\kappa$-basis  algorithm}\label{basis}
	\begin{algorithmic}[1]
		\Procedure{\small $\beta_\kappa$-Basis}{}
			\While{each dimension $n$}  
	\State  initialize $M_n=K_n=\mathcal{B}_{K_n}=\varnothing$
	\State arrange $n$-simplices $\sigma_1,\sigma_2,\ldots,\sigma_l$ with $\omega(\sigma_{i})\geq \omega(\sigma_{i+1})$
	 \For{$0\leq m\leq l-1$}
	\State solve the  equation  of 
 $r_j$'s for   $r_j \in \mathbb{F}$:  $$\partial_n   (r_{m+1}\sigma_{m+1} +\sum_{\sigma_j\in M_n}r_j \sigma_j )=0$$ 
	\If{non-trivial solution of  $r_j$'s  exists} 
	\State  set $
	\kappa \gets \sigma_{m+1} 
	$ and $
	\beta_\kappa \gets \sum_{\sigma_j\in M_n} \frac{r_j}{r_{m+1}}\sigma_j +\sigma_{m+1} 
	$
	\State    update $K_n \gets K_n \cup \{ \kappa=\sigma_{m+1} \}$ and $\mathcal{B}_{K_n}\gets \mathcal{B}_{K_n} \cup \{ \beta_\kappa\}$
  \Else 
  \State update $M_n\gets M_n \cup \{ \mu_{m+1}=\sigma_{m+1} \}$

	\EndIf
	\EndFor
		\EndWhile

		\State \small  \textbf{Return} $\beta_\kappa$-basis $\mathcal{B}_{K_n} $ together with   $\{\partial_{n}(\mu_j)| \mu_j\in M_n\}$ for each dimension $n$
		\EndProcedure
	\end{algorithmic}
\end{algorithm}

\begin{algorithm}
	\caption{Bijection $\Phi$  construction}\label{couple}
	\begin{algorithmic}[1]
		\Procedure{\small $\Phi$-couples}{}
		\While{each dimension $n$}  
		\State  Input $\beta$-basis $\mathcal{B}_{K_n} $ and $\{\partial_{n+1}(\mu_j)| \mu_j\in M_{n+1}\}$
		\State arrange $n$-simplices  $\kappa_1,\kappa_2,\ldots,\kappa_q$ with $\omega(\kappa_{i})\leq \omega(\kappa_{i+1})$
		\State arrange $n+1$-simplices $\mu_1,\mu_2,\ldots,\mu_p$ with $\omega(\mu_{j})\geq \omega(\mu_{j+1})$
		\State  construct the corresponding projection matrix $Q\in \mathbb{F}_{p\times q}$
			 \For{each column $1\leq k\leq q$}
		\State  Find out the minimum row $j_k$ having  nonzero entry  in column $k$  
			\If{$j_k$ exists}
		\State  pair  $\Phi_n (\kappa_k)=\mu_{j_k} $ 
		\State  update $Q$ by eliminating any other nonzero entry  in column $k$ via subtracting  row $j_k$
		\State update $Q$ by deleting any other nonzero entries  in  row $j_k$ 		
		\EndIf
			\EndFor
			\EndWhile 
		\State \small  \textbf{Return} the coupling $\Phi$ for each dimension $n$
		\EndProcedure
	\end{algorithmic}
\end{algorithm}

\section{Software and Implementation} \sloppy
Using the insights developed in the previous sections we have implemented the python package \texttt{WeightedSimplicialHomology}, which ingests a weighted simplical complex as a specifically formatted text file and outputs its weighted homolgy at all dimensions, see Figure~\ref{fig:package example}.

The module consists of two classes: \texttt{SimplicialConstructors}, which allows the user to specify a weighted simplicial complex either through its list of simplices or by specifying its maximal simplices (in this latter case all of the simplex weights are the same), and \texttt{mWSC} which implements the algorithms of section 4. The results of section 3 allow one to bypass computing the Smith Normal Form of a boundary matrix with entries in $\mathbb{F}[[\pi]]$ and instead reduces the problem to linear algebra over the residue field $\mathbb{F}$. \texttt{WeightedSimplicialHomology} takes advantage of this observation and leverages the already extant linear algebra capabilities of NumPy in this regard. 

The method \texttt{compute\_K\_basis} implements algorithm \ref{basis}. Checking for a nontrivial solution to the matrix equation~(\ref{Eq:beta}) is accomplished by iteratively examining columns of the boundary matrix and verifying at each step that the rank is less than the number of columns, which indicates the existence of such a nontrivial solution. Once said solution is known to exist it is computed using \texttt{numpy.linalg.lstsq}, which returns a least-squares result. This method is applied to the equation 
\[
\partial_{n}\Big(\sum_{\sigma_{j}\in M_{m}}r_{j}\sigma_{j} \Big) = \sigma_{m+1}.
\]
Since we have determined there is a unique (up to scaling) solution to this equation in the above step, \texttt{lstsq} will return precisely this solution, which is then added to the basis under construction.

The method \texttt{find\_basis\_pairing} implements algorithm \ref{couple}. We construct the matrix representing the projection of $\Ima \partial_{n+1}$ onto $\ker \partial_{n}$ spanned by the $\beta$-basis $\mathcal{B}_{K_{n}}$. For each column of $Q$ we find the first non-zero entry and use this as a pivot in elementary row operations to eliminate all other entries in its column. After each reduction step, as a precaution, we opt to zero out any entries smaller in absolute value than a fixed tolerance $\epsilon$, where 
$\epsilon$ is some multiple of the machine precision - in practice however we have not yet observed any difference in behavior resulting from this step. 

Our initial attempt at constructing this pipeline proceeded by directly computing the Smith Normal Form of
$\partial_{\ast}^{v}$ with coefficients in $\mathbb{F}[x]_{(x)}$ - the polynomial ring $\mathbb{F}[x]$ localized at the $x$ indeterminate. The computation resulted in ballooning of both the coefficients of polynomials and of their degrees, which made the computation unfeasible even for matrices as small as $7 \times 7$. Having circumvented the need to directly represent elements of the DVR by our theoretical considerations, and having efficient Python native linear algebra tools available, \texttt{WeightedSimplicialHomology} no longer suffers the same ballooning in run-time and memory. A such, we are able to compute the homology of canonical examples such as $T^{2}$, with a triangulation using 17 $2$-simplices, and $\mathbb{R}P^{2}$, with a triangulation using 9 $2$-simplices, in less than a second. This suggests that that \texttt{WeightedSimplicialHomology} will be able to scale up to medium-sized complexes while still running in a moderate amount of time (this will be the purview of an upcoming benchmarking paper). 

\texttt{WeightedSimplicialHomology} currently computes weighted homology over the ring $\mathbb{Q}[[\pi]]$. Future iterations will have the ability to compute weighted homology over $\mathbb{F}_{p}[[\pi]]$ for $p$ prime, and possibly over $\mathbb{F}_{q}[[\pi]]$ for $q = p^{k}$. The source code for this project can be found at \url{https://gitlab.com/Nerlf/WeightedSimplicialHomology}.

\begin{figure}
    \centering
    \includegraphics[scale = 0.425]{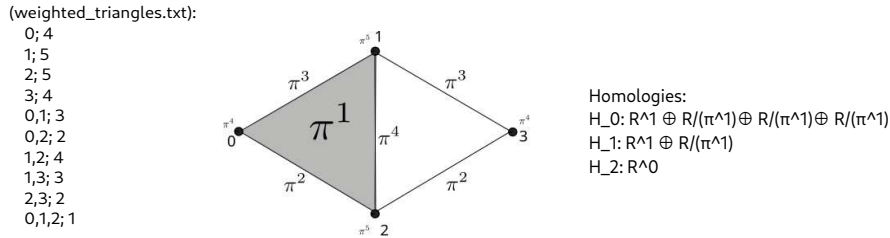}
    \caption{Left: input file for specifying a non-trivially weighted complex using \texttt{Simplicial\_File\_Reader;}
    Center: The complex specified by the file, consisting of a filled triangle and an empty triangle joined along an edge;
    Right: The weighted homology of the weighted complex is printed using \texttt{print\_Homology}, of note is the presence of nontrivial 
    torsion in degrees $0$ and $1$.}
    \label{fig:package example}
\end{figure}

\section*{Acknowledgments}
This work was supported by the VDH Grant PV-BII VDH COVID-19 Modeling Program VDH-21-501-0135.

\bibliographystyle{elsarticle-harv} 

\end{document}